\documentclass{amsart}

\usepackage{amssymb}
\usepackage{url}
\usepackage[cmtip,all]{xy}

\newtheorem{theorem}{Theorem}[section]
\newtheorem{lemma}[theorem]{Lemma}
\newtheorem{proposition}[theorem]{Proposition}
\newtheorem{corollary}[theorem]{Corollary}

\theoremstyle{definition}
\newtheorem{definition}[theorem]{Definition}

\numberwithin{equation}{section}

\newcommand{\de}{\colon}
\newcommand{\st}{\mid}
\newcommand{\C}{\mathbb{C}}
\newcommand{\N}{\mathbb{N}}
\newcommand{\Q}{\mathbb{Q}} 
\newcommand{\Ar}{\mathcal{A}}
\newcommand{\Br}{\mathcal{B}}

\DeclareMathOperator{\Poin}{Poin}

\DeclareMathOperator{\codim}{codim}
\DeclareMathOperator{\sgn}{sgn}
\renewcommand{\phi}{\varphi}

\DeclareMathOperator{\id}{id}
\newcommand{\equivr}{\sim}

\title[Long exact sequence for deleted and restricted arrangements]{A long exact sequence in cohomology for deleted and restricted subspaces arrangements}

\author{Gery Debongnie}
\address{UCL, Departement de mathematique \\ Chemin du Cyclotron, 2 \\ 
B-1348 Louvain-la-neuve \\ Belgium}
\email{debongnie@math.ucl.ac.be}
\thanks{The author is an ``Aspirant'' of the ``Fonds National pour la 
Recherche Scientifique'' (FNRS), Belgium.}

\subjclass[2000]{Primary 55P62}

\begin{document}

\begin{abstract}
The notions of \emph{deleted} and \emph{restricted} arrangements have been very useful in the study of arrangements of hyperplanes. Let $\Ar$ be an arrangement of hyperplanes and $x \in \Ar$. The deleted and restricted arrangements $\Ar'$ and $\Ar''$ allows us to compute recursively the Poincar\'e polynomials of the complement space $M(\Ar)$ with the following formula~: \[
\Poin(M(\Ar), t) = \Poin(M(\Ar'), t) + t \Poin(M(\Ar''), t).
\]
In this paper, we consider the extension of this formula to arbitrary subspaces arrangements. The main result is the existence of a long exact sequence connecting the rational cohomology of $M(\Ar)$, $M(\Ar')$ and $M(\Ar'')$. Using this sequence, we obtain new results connecting the Betti numbers and Poincar\'e polynomials of deleted and restricted arrangements.
\end{abstract}

\maketitle

\section{Introduction}
An arrangement of subspaces is a finite set $\Ar = \{x_0, \dotsc, x_n\}$ of vector subspaces in $\C^l$. Unless said otherwise, we'll always suppose that $\Ar$ is a complex arrangement and that there is no $x \neq y$ in $\Ar$ such that $x \subset y$. We consider the poset of intersections of elements of $\Ar$ ordered by reverse inclusion.  We define two operations on $L(\Ar)$~: the meet $x \wedge y = \cap\{z \in L(\Ar) \st x \cup y \subset z\}$ and the join $x \vee y = x \cap y$. These two operations endow $L(\Ar)$ with a lattice structure.  Our main interest is the study of the complement space $M(\Ar) = \C^l \setminus \cup_{x \in \Ar} x$.

When all the subspaces $x \in \Ar$ are hyperplanes ($\codim x = 1$), then $\Ar$ is an arrangement of hyperplanes.  If $\Ar$ is such an arrangement and $x \in \Ar$, then we can consider the arrangements $\Ar' = \Ar \setminus \{x\}$ and $\Ar'' = \{x \cap y \st y \in \Ar''\}$. The triple $(\Ar, \Ar', \Ar'')$ has been studied (a good reference is the book of P. Orlik and H. Terao~\cite{orl92}), and possesses some interesting properties.  For example, there is a relationship between the Poincar\'e polynomials of $M(\Ar), M(\Ar')$ and $M(\Ar'')$~: \[
\Poin(M(\Ar), t) = \Poin(M(\Ar'), t) + t \Poin(M(\Ar''), t).
\]

The goal of this paper is to generalize that property to arbitrary subspaces arrangements. But before doing that, we need to define what are the deleted and restricted arrangements $\Ar'$ and $\Ar''$ in the general case.  The deleted arrangement is obvious, but there is a slight difficulty when we want to extend restricted arrangements. Section~\ref{sec:drsa} contains a proper definition of $\Ar''$. It also gives a description of a simple rational model $D(\Ar)$ of $M(\Ar)$.

Let $D(\Ar), D(\Ar')$ and $D(\Ar'')$ be the rational models of $M(\Ar)$, $M(\Ar')$ and $M(\Ar'')$, as defined in section~\ref{sec:drsa}.  From the definition, there exists an injective map (the inclusion) $j \de D(\Ar') \to D(\Ar)$. This map induces a short exact sequence of cochains complexes \[
0 \to D(\Ar') \stackrel{j}{\to} D(\Ar) \to  \frac{D(\Ar)}{D(\Ar')} \to 0.
\]

In section~\ref{sec:quasiiso}, we will show that there is a c.d.g.a. $D(\widetilde{\Ar''})$ and quasi-isomorphisms \[
\frac{D(\Ar)}{D(\Ar')} \stackrel{\simeq}{\to} D(\widetilde{\Ar''}) \stackrel{\simeq}{\leftarrow} D(\Ar'').
\]

Together, these results give us a long exact sequence connecting the cohomology of $M(\Ar)$, $M(\Ar')$ and $M(\Ar'')$, which is studied in section~\ref{sec:drt} (see theorem~\ref{theo:lesic}). A few consequences are studied as well. For example, the Euler characteristic of the complement space of a complex subspaces arrangement is always 0.  Under some conditions, this long exact sequence allows us to generalize the formula about the Poincar\'e polynomials to \[
\Poin(M(\Ar), t) = \Poin(M(\Ar'), t) + t^{\deg(x)}\Poin(M(\Ar''), t).
\]

Finally, section~\ref{sec:agl} gives stronger results for the specific case of subspaces arrangements with a geometric lattice.

\section{Rational models and restricted subspaces arrangements} \label{sec:drsa}

In~\cite{yu05}, S. Yuzvinsky and E. Feichtner described a
simple rational model for the topological space $M(\Ar)$ of arbitrary subspaces arrangements.  This model is a graded differential algebra denoted by $D(\Ar)$. 

Let $\Ar = \{x_0, \dotsc, x_n\}$ be a subspaces arrangement. A rational model for the space $M(\Ar)$ is given as follows :  choose a linear  order on $\Ar$ (in this paper, we'll use $x_0 < x_1 < \dotso < x_n$).  The graded vector space $D(\Ar)$ has a basis given by  all
subsets $\sigma \subseteq \Ar$.  For $\sigma = \{x_{i_1}, \dotsc, x_{i_r}\}$, we
define a differential by the formula \[
d\sigma = \sum_{j:\vee(\sigma \setminus \{x_{i_j}\}) = \vee \sigma} (-1)^j
(\sigma\setminus\{x_{i_j}\})
\]
where the indexing of the elements  in $\sigma$ follows the linear order
imposed on $\Ar$. With $\deg(\sigma) = 2 \codim \vee \sigma - |\sigma|$,
$D(\Ar)$  is a  cochain complex.   Finally,  the multiplication is defined, for $\sigma, \tau \subseteq \Ar$, by
\[
\sigma \cdot \tau = \left\{\begin{aligned} (-1)^{\sgn\epsilon(\sigma, 
\tau)} \sigma \cup \tau &\qquad \text{ if } \codim \vee \sigma + \codim \vee 
\tau = \codim \vee(\sigma \cup \tau), \\ 0 & \qquad\text{ otherwise,} 
\end{aligned}\right.
\]
where  $\epsilon(\sigma, \tau)$ is the  permutation  that, applied  to $
\sigma \cup  \tau$ with the  induced linear  order, places  elements  of
$\tau$  after elements  of $\sigma$, both  in the induced  linear order.

Since we are interested in deleted and restricted subspaces arrangements, the rest of this section will be used to define such arrangements and to better understand their models. Let $\Ar$ be a (non-empty) subspaces arrangements and $x_0 \in \Ar$. Obviously, the definition of the \emph{deleted arrangement} is $\Ar' = \Ar \setminus \{x_0\}$.  Its model $D(\Ar')$ is a subalgebra of $D(\Ar)$ (the inclusion is an injective map). 

However, the restricted arrangement is a little trickier to define for subspaces arrangements. Let $\widetilde{\Ar''} = \{ x_0 \cap y \st y \in \Ar' \}$.  The difficulty lies in the fact that $\widetilde{\Ar''}$ might contain two distinct subspaces $x_0 \cap y_1$ and $x_0 \cap y_2$ such that $x_0 \cap y_1 \subsetneqq x_0 \cap y_2$. Since we are mainly interested in the complement space, it makes sense to define the \emph{restricted arrangement} (in $x_0$) by \[
\Ar'' = \widetilde{\Ar''} \setminus \{ x_0 \cap y \st y \in \Ar' \text{ and } \exists y' \in \Ar' \text{ with } y \neq y' \text{ and } x_0 \cap y \subset x_0 \cap y'  \}.
\]

It makes sense to consider the graded differential algebra $D(\widetilde{\Ar''})$, even if $\widetilde{\Ar''}$ is not a proper subspaces arrangement.  Later in this section, the corollary~\ref{cor:artilde} states that the natural inclusion $D(\Ar'') \hookrightarrow D(\widetilde{\Ar''})$ is a quasi-isomorphism.

The key to this corollary~\ref{cor:artilde} is contained in the next proposition. 
\begin{proposition} \label{prop:subar}
Suppose $\Br$ is a subspaces arrangement with two subspaces $u$ and $v$ such that $u \subset v$. Let $\Br' = \Br \setminus \{u \}$. Then the inclusion $D(\Br') \hookrightarrow D(\Br)$ is a quasi-isomorphism.
\end{proposition}
\begin{proof}
The injective map $D(\Br') \to D(\Br)$ gives us the following short exact sequence of cochain complexes~: \[
0 \to D(\Br') \to D(\Br) \to \frac{D(\Br)}{D(\Br')} \to 0.
\]
Since there is a long exact sequence in cohomology, it is sufficient to prove that $H^\star(D(\Br)/D(\Br')) = 0$. The cochain complexes $D(\Br)/D(\Br')$ is a vector space generated by all the elements of the form
\begin{itemize}
\item $\sigma = \{u, x_{i_1}, \dotsc, x_{i_r}\}$ with $x_{i_j} \not \in \{u,v\}$. The differential $d\sigma$ is a sum of some $ \pm \{u, x_{i_1}, \dotsc, \widehat{x_{i_j}}, \dotsc, x_{i_r}\}$.
\item $\sigma = \{u, v, x_{i_1}, \dotsc, x_{i_r}\}$ with $x_{i_j} \not \in \{u, v\}$. The differential $d\sigma$ is a sum of $\{u, x_{i_1}, \dotsc, x_{i_r}\}$ and some elements $\pm \{u,v, \dotsc, \widehat{x_{i_j}}, \dotsc, x_{i_r}\}$.
\end{itemize}

 Let's introduce a new grading on $D(\Br)/D(\Br')$. We say that $\deg\{u, x_{i_1}, \dotsc, x_{i_r}\} = r$ and $\deg\{u,v, x_{i_1}, \dotsc, x_{i_r}\} = r + 1$. The differential decreases this degree by one, and we have a chain complex. 

Let's introduce a filtration on $D(\Br)/D(\Br')$~: $F_0$ is the subcomplex generated by $\{u\}$ and $\{u,v\}$, and in general, $F_r$ is generated by all the $\{u, x_{i_1}, \dotsc, x_{i_s}\}$ and $\{u,v, x_{i_1}, \dotsc, x_{i_s}\}$ with $s \leq r$.  So, we have an increasing filtration $F_0 \subseteq F_1 \subseteq \dotso \subseteq F_r \subseteq \dotso$ Let's define the filtration degree $d_0$ of $\sigma$ by the number of $x_{i_j}$ in $\sigma$, and the complementary degree is 0 if $v \not \in \sigma$ and $1$ if $v \in \sigma$. This gives us a spectral sequence. The page $E^0$ is~:
\begin{itemize}
\item $E^0_{p, 0}$ is a vector subspace generated by all the $\{u, x_{i_1}, \dotsc, x_{i_p}\}$.
\item $E^0_{p, 1}$ is generated by all the $\{u,v, x_{i_1}, \dotsc, x_{i_p}\}$.
\item $E^0_{p, q}$ is zero if $q \geq 2$.
\end{itemize}
The differential $d_0\de E^0_{p,q} \to E^0_{p, q-1}$ is zero for every $q \neq 1$. For $q = 1$, $d_0 \de E^0_{p, 1} \to E^0_{p, 0}$ maps $\{u,v, x_{i_1}, \dotsc, x_{i_p}\}$ to $\{u, x_{i_1}, \dotsc, x_{i_p}\}$. So, it is obvious that every spaces $E^1_{p,q}$ of the next page is simply zero.  Therefore, there is no homology in $D(\Br)/D(\Br')$, which completes the proof.
\end{proof}

\begin{corollary} \label{cor:artilde}
Let $\Ar$ be a subspaces arrangement. Then the inclusion $D(\Ar'') \hookrightarrow D(\widetilde{\Ar''})$ is a quasi-isomorphism.
\end{corollary}
\begin{proof}
It is a consequence of the definitions and proposition~\ref{prop:subar}.
\end{proof}

\section{A quasi-isomorphism connecting $D(\Ar)/D(\Ar')$ and $D(\widetilde{\Ar''})$} \label{sec:quasiiso}

The goal of this section will be to show the existence of a quasi-isomorphism $\frac{D(\Ar)}{D(\Ar')} \to D(\widetilde{\Ar''})$.  This will give us a long exact sequence connecting the rational cohomology of $M(\Ar), M(\Ar')$ and $M(\Ar'')$.

But before that, we need to define a few things. There is an equivalence relation in $\Ar'$ : $x_i \equivr x_j$ if and only if $x_0 \cap x_i = x_0 \cap x_j$. This equivalence relation gives a partition of $\Ar' = \Ar_1 \cup \Ar_2 \cup \dotso \cup \Ar_r$. In order to get the signs right, the linear order on the elements of $\Ar$ is chosen with the property that if $i < j$, $y \in \Ar_i$ and $z \in \Ar_j$, then $y < z$.

Let $E = \{ (y,z) \in \Ar' \times \Ar' \st y \equivr z \text{ and } y \neq z \}$. Let's define the map $\phi \de D(\Ar)\to D(\widetilde{\Ar''})$. Let $\sigma \subseteq \Ar$. If $x_0 \not \in \sigma$ or if $x_0 \in \sigma$ and there exists $(y,z) \in E$ such that $\{y,z\} \in \sigma$, then $\phi(\sigma) = 0$. Otherwise, $\sigma = \{x_0, x_{i_1}, \dotsc, x_{i_r}\} \subseteq \Ar$ is sent to $\phi(\sigma) = (-1)^r \{x_0 \cap x_{i_1}, \dotsc, x_0 \cap x_{i_r} \} \subseteq \widetilde{\Ar''}$ (in particular, $\phi(\{x_0\}) = \emptyset$). This map is not multiplicative but it commutes with the differential.

\begin{lemma}
The map $\phi$ is such that $\phi d = d \phi$.
\end{lemma}
\begin{proof}
Let $\sigma \subseteq \Ar$. We need to check the following situations.
\begin{itemize}
\item If $x_0 \not\in \sigma$, then $d\phi \sigma = 0 = \phi d \sigma$. 
\item if $x_0 \in \sigma$ and there exists $(y,z) \in E$ with $\{y, z\} \in \sigma$, then $d \phi \sigma = 0$ and 
\begin{multline*}
d\{x_0,y,z, x_{i_1}, \dotsc, x_{i_r} \} = \alpha \left(\sigma \setminus \{x_0 \} \right) + \{x_0, z, x_{i_1}, \dotsc, x_{i_r}\} \\ - \{x_0, y, x_{i_1}, \dotsc, x_{i_r} \} + \sum_j \alpha_j \left(\sigma \setminus\{x_{i_j}\} \right), 
\end{multline*}
where $\alpha$ and the $\alpha_j$ can be $0$ or $\pm 1$. So, by definition of $\phi$, we have : $\phi d\sigma = 0$.
\item Finally, the last possibility is that $x_0 \in \sigma$ and no $(y,z) \in E$ is such that $\{y,z\} \subset \sigma$. In that case, let $\sigma = \{x_0, x_{i_1}, \dotsc, x_{i_r}\}$ and denote by $J_\sigma$ the set $\{j \st \cap_{m = 1, m \neq j}^r (x_{i_m} \cap x_0) = \cap_{m=1}^r (x_{i_m} \cap x_0) \}$. We have
\begin{align*}
\phi(d \sigma) & = \phi\left( \alpha \left(\sigma \setminus \{x_0\}\right) + \sum_{j \in J_\sigma} (-1)^{j+1}\sigma \setminus \{x_{i_j}\}\right) \\ &= (-1)^{r-1} \sum_{j \in J_\sigma} (-1)^{j+1} \{x_0 \cap x_{i_1}, \dotsc, \widehat{x_0 \cap x_{i_j}}, \dotsc, x_0 \cap x_{i_r} \}, \\
d \phi(\sigma) &= (-1)^r d\{ x_0 \cap x_{i_1}, \dotsc, x_0 \cap x_{i_r} \} \\
 &= (-1)^r \sum_{j \in J_\sigma} (-1)^j \{x_0 \cap x_{i_1}, \dotsc, \widehat{x_0 \cap x_{i_j}}, \dotsc, x_0 \cap x_{i_r} \}.
\end{align*}
\end{itemize}
So, the map $\phi$ commutes with the differential $d$.
\end{proof}

The map $\phi$ is surjective.  Obviously, it satisfies  $\phi j = 0$ and induces a surjective map \[
\bar{\phi} \de \frac{D(\Ar)}{D(\Ar')} \to D(\widetilde{\Ar''})
\]
We'll show that this map induces an isomorphism in cohomology.  If the set $E$ is empty, then $\bar{\phi}$ is injective and is an isomorphism. Otherwise, for every $(u,v) \in E$, let $I_{u,v}$ be the vector subspace of $D(\Ar)/D(\Ar')$ generated by all the elements of the form $\{x_0,u, y_{j_1}, \dotsc, y_{j_r} \} - \{x_0, v, y_{j_1}, \dotsc, y_{j_r} \}$ and $\{x_0, u, v, y_{j_1}, \dotsc, y_{j_r} \}$ with the $y_i \in \Ar \setminus \{x_0,u,v\}$.  Let $V = \sum_{(u,v) \in E} I_{u,v}$. We have the following lemmas~:
\begin{lemma} \label{lem:vkernel}
In $D(A)/D(A')$, we have $V = \ker \bar{\phi}$.
\end{lemma}
\begin{proof}
It is clear that $V \subseteq \ker \bar{\phi}$. Let $u = \sum_{s \in I} \alpha_s \sigma_s \in \ker \bar{\phi}$, with $\alpha_s \neq 0 \; \forall s \in I$ and $\sigma_s \neq \sigma_t$ if $s \neq t$.  To show that $u \in V$, we will show that we can substract elements of $V$ from $u$ until we obtain zero.
\begin{itemize}
\item If there exists a $s \in I$ such that $\sigma_s$ is of the form $\{x_0, u, v, y_{j_1}, \dotsc, y_{j_r}\}$ with $x_0 \cap u = x_0 \cap v$ and $u \neq v$, then $\alpha_s \sigma_s \in V$ and we can substract it from $u$. Let's repeat this operation until we obtain a sum $\sum_{s \in I'} \alpha_s \sigma_s$ without any such $\sigma_s$.
\item Let $t \in I'$. The element $\sigma_t$ can only be of the form $\{x_0, x_{i_1}, \dotsc, x_{i_r}\}$ with $x_0 \cap x_{i_j} \neq x_0 \cap x_{i_k}$ for $j \neq k$. We have \[
0 = \bar{\phi}\left(\sum_{s \in I'} \alpha_s \sigma_s\right) = (-1)^r \alpha_t \{x_0 \cap x_{i_1}, \dotsc, x_0 \cap x_{i_r}\} + \sum_{s \in I' \setminus \{t\}} \alpha_s \bar{\phi}(\sigma_s).
\]
So, there exists a $s \in I' \setminus\{t\}$ such that $\bar{\phi}(\sigma_s) = \bar{\phi}(\sigma_t)$. This is possible only if $\sigma_s = \{x_0, y_1, \dotsc, y_r\}$ with $x_0 \cap x_{i_j} = x_0 \cap y_j$. In that case, the difference $\sigma_s - \sigma_t$ can be rewritten as \[
\sum_{j \st y_j \neq x_{i_j}} \{x_0, y_1, \dotsc, y_{j-1}, y_j, x_{i_{j+1}}, \dotsc, x_{i_r} \} - \{x_0, y_1, \dotsc, y_{j-1}, x_{i_j}, x_{i_{j+1}}, \dotsc, x_{i_r}\},
\]
which is in $V$. We can substract $\alpha_t(\sigma_t - \sigma_s)$ from the sum and we get another sum $\sum_{s \in I''} \alpha_s \sigma_s$ such that $|I''| < |I'|$. This process can be repeated until we obtain an empty sum.
\end{itemize}
Therefore, $u \in V$ and $V = \ker \bar{\phi}$.
\end{proof}

Let $(u,v) \in E$. It is clear that $d(I_{u,v}) \subset I_{u,v}$, so $I_{u,v}$ form a subcomplex of $D(\Ar)/D(\Ar')$.
\begin{lemma} \label{lem:iacyclic}
For every $(u,v) \in E$, the complex $I_{u,v}$ is acyclic.
\end{lemma}
\begin{proof}
Let's consider a  cocycle of the form \[
\alpha = \sum_i \alpha_i\left(\{x_0, u, y_{i_1}, \dotsc, y_{i_r}\} - \{x_0, v, y_{i_1}, \dotsc, y_{i_r}\} \right) + \sum_j \beta_j \{x_0, u, v, z_{j_1}, \dotsc, z_{j_{r-1}}\},
\]
with the $y_{i_k}, z_{j_k} \not \in \{x_0, u, v\}$. We want to show that $\alpha$ is a coboundary. Let 
\begin{align*}
K_i &= \big\{k \in \{i_1, \dotsc, i_r\} \st x_0 \cap u \cap y_{i_1} \cap \dots \cap y_{i_r} = \\
& \qquad \qquad \qquad x_0 \cap u \cap y_{i_1} \cap \dots \cap \widehat{y_{i_k}} \dots \cap y_{i_r}\big\}, \\
L_j &= \big\{  k \in \{j_1, \dotsc, j_{r-1}\} \st x_0 \cap u \cap v \cap z_{j_1} \cap \dots \cap z_{j_{r-1}} = \\ 
& \qquad \qquad \qquad x \cap u \cap v \cap z_{j_1} \cap \dots \cap \widehat{z_{j_k}} \dots \cap z_{j_{r-1}}  \big\}.
\end{align*}
Since $\alpha$ is a cocycle, we have~: 
\begin{align*}
0 = d\alpha =& \sum_i \alpha_i\sum_{k \in K_i} (-1)^k \big( \{x_0, u, y_{i_1}, \dotsc \widehat{y_{i_k}} \dotsc, y_{i_r}\} - \{ x_0, v, y_{i_1}, \dotsc \widehat{y_{i_k}} \dotsc, y_{i_r}\}\big) \\
 & + \sum_j \beta_j\big(\{x_0, u, z_{j_1}, \dotsc, z_{j_{r-1}} \} - \{x_0, v, z_{j_1}, \dotsc, z_{j_{r-1}} \} \big) \\
 & + \sum_j \beta_j \sum_{k \in L_j} (-1)^{k+1} \{x_0, u, v, z_{j_1}, \dotsc \widehat{z_{j_k}} \dotsc, z_{j_{r-1}} \}.
\end{align*}
It implies that \[
\sum_i \alpha_i\sum_{k \in K_i} (-1)^k \{x_0, u, v, y_{i_1}, \dotsc \widehat{y_{i_k}} \dotsc, y_{i_r}\} + \sum_j \beta_j \{ x_0, u, v, z_{j_1}, \dotsc, z_{j_{r-1}} \} = 0.
\]
Let's consider $\omega =  - \sum_i \alpha_i \{x_0, u, v, y_{i_1}, \dotsc, y_{i_r}\}$. We have 
\begin{align*}
d\omega &=  \sum_i \alpha_i \big( \{x_0, u, y_{i_1}, \dotsc, y_{i_r}\} - \{ x_0, v, y_{i_1}, \dotsc, y_{i_r}\}\big) \\
 & \qquad - \sum_i \alpha_i\sum_{k \in K_i} (-1)^k \{x_0, u, v, y_{i_1}, \dotsc \widehat{y_{i_k}} \dotsc, y_{i_r}\} \\
 & = \alpha.
\end{align*}
Hence, $\alpha$ is a coboundary and $I_{u,v}$ is acyclic.
\end{proof}

\begin{lemma} \label{lem:acyclic}
The vector space $V$ is such that $d(V) \subseteq V$ and $V$ is acyclic.
\end{lemma}
\begin{proof}
It is obvious that $d(V) \subseteq V$. Showing that $V$ is acyclic is technical.  We'll do it in three steps. First, we choose $r$ elements $(u_i, v_i) \in E$ that have some nice properties. Then, using these elements, we construct a sequence of vector subspaces $I_0 \subsetneqq I_1 \subsetneqq I_r = V$. Finally, we prove by induction that all these subspaces are acyclic.

Let's remember that the linear order respects the equivalence relation $\equivr$ on $\Ar'$ in the following sense~: the equivalence classes are $\Ar' = \Ar_1 \cup \Ar_2 \cup \dotso \cup \Ar_r$ and if $i < j$, $y \in \Ar_i$ and $z \in \Ar_j$, then $y < z$. Let $\Ar_1 = \{x_{j_1}, \dotsc, x_{j_2 - 1}\}$, $\Ar_2 = \{x_{j_2}, \dotsc, x_{j_3 - 1}\}, \dotsc, \Ar_r = \{x_{j_r}, \dotsc, x_{j_{r+1} - 1}\}$.  Let $\Ar_i$ be the first equivalence class with more than 1 element and $(u_0, v_0) = (x_{j_i}, x_{j_i + 1})\in E$, $(u_1, v_1) = (u_0, x_{j_i + 2}), \dotsc, (u_m, v_m) = (u_0, x_{j_{i+1} - 1})$.  Let $\Ar_k$ be the next equivalence class with more than 1 element and $(u_{m+1}, v_{m+1}) = (x_{j_k}, x_{j_k+1}), (u_{m+2}, v_{m+2}) = (x_{j_k}, x_{j_k + 2}), \dotsc$  We can do this until we have a sequence $(u_i, v_i)_{0 \leq i \leq r} \subseteq E$ with the following properties~: $V = \sum_{i=0}^r I_{u_i, v_i}$ and $v_n \not\in \{u_0, v_0, u_1, v_1, \dotsc, u_{n-1}, v_{n-1} \}$.

Let $I_0 = I_{u_0, v_0}$ and $I_j = \sum_{i = 0}^j I_{u_i, v_i}$. The properties of the $(u_i, v_i)$ implies that we have an increasing sequence \[
I_{u_0, v_0} =I_0 \subsetneqq I_1 \subsetneqq \dotso \subsetneqq I_r = V. \]
Lemma~\ref{lem:iacyclic} shows that $I_0$ is acyclic. Suppose $I_{n-1}$ acyclic and let's show that $I_n = I_{n-1} + I_{u_n, v_n}$ is acyclic as well.  

The key observation is that every element of $I_{n-1}$ can be written as a sum $\alpha_1 + \alpha_2 + d\omega$ where $\alpha_1 \in I_{n-1}, \alpha_2 \in I_{u_n, v_n}$, $\omega \in I_{n-1}$ and there is no $v_n$ in any of the terms of $\alpha_1$.  It is sufficient to prove this for every element of a generating sequence of $I_{n-1}$~:
\begin{itemize}
\item Let $\sigma = \{x_0, u_i, v_i, \dotsc \}$. If $v_n \in \sigma$ and $u_n \in \sigma$, then $\sigma = 0 + \sigma + d(0)$. If $v_n \in \sigma$ and $u_n \not \in \sigma$, then $d(\sigma \cup \{u_n\}) = \pm \sigma \pm (\sigma \cup \{u_n\}) \setminus \{v_n\} + S$ where $S$ is a sum with $\{u_n, v_n \}$ in each term, which means that $S$ is in $I_{u_n, v_n}$. It shows that we have the required decomposition $\sigma = \pm (\sigma  \cup \{u_n \}) \setminus \{v_n\} \pm S \pm d(\sigma \cup \{u_n\})$.
\item Let's consider $\alpha = \{x_0, u_i, y_{j_1}, \dotsc, y_{j_r}\} - \{x_0, v_i, y_{j_1}, \dotsc, y_{j_r}\}$. If $v_n \in \{y_{j_1}, \dotsc, y_{j_r}\}$ and $u_n \not\in \{y_{j_1}, \dotsc, y_{j_r}\}$, then form $\omega = \{x_0, u_i, y_{j_1}, \dotsc, y_{j_r} \} \cup \{ u_n \} - \{x_0, v_i, y_{j_1}, \dotsc, y_{j_r}\} \cup \{u_n\}$.  As in the first case, we obtain a good decomposition from $d\omega$.
\end{itemize}
This observation allow us to show that every cocycle in $I_n$ is a coboundary. Let $\alpha \in I_n$ such that $d\alpha = 0$. Since $I_n = I_{n-1} + I_{u_n, v_n}$, we can write $\alpha = \alpha_1 + \alpha_2 + d\omega$ with $\alpha_1 \in I_{n-1}, \alpha_2 \in I_{u_n, v_n}$, $\omega \in I_n$ and no $v_n$ in any term of $\alpha_1$. The vector $\alpha_2$ can be written \[
\alpha_2 = \sum_i \gamma_i\left(\{x_0, u_n, y_{i_1}, \dotsc, y_{i_r}\} - \{x_0, v_n, y_{i_1}, \dotsc, y_{i_r}\} \right) + \sum_j \beta_j \{x_0,u_n,v_n, z_{j_1}, \dotsc, z_{j_{r-1}}\}.
\]
For appropriate sets $K_i$ and $L_j$ (as in lemma~\ref{lem:iacyclic}), its differential is
\begin{align*}
d\alpha_2 =& \sum_i \gamma_i\sum_{k \in K_i} (-1)^k \big( \{x_0, u_n, y_{i_1}, \dotsc \widehat{y_{i_k}} \dotsc, y_{i_r}\} - \{ x_0, v_n, y_{i_1}, \dotsc \widehat{y_{i_k}} \dotsc, y_{i_r}\}\big) \\
 & + \sum_j \beta_j\big(\{x_0,u_n, z_{j_1}, \dotsc, z_{j_{r-1}} \} - \{x_0, v_n, z_{j_1}, \dotsc, z_{j_{r-1}} \} \big) \\
 & + \sum_j \beta_j \sum_{k \in L_j} (-1)^{k+1} \{x_0, u_n, v_n, z_{j_1}, \dotsc \widehat{z_{j_k}} \dotsc, z_{j_{r-1}} \}.
\end{align*}
But there is no term in $\alpha_1$ containing $v_n$ and $d\alpha_1 + d\alpha_2 = 0$. Therefore, $d\alpha_1 = d\alpha_2 = 0$.  By the induction hypothesis, there exists a $\omega_1$ such that $d\omega_1 = \alpha_1$ and by lemma~\ref{lem:iacyclic}, there exists a $\omega_2$ such that $d\omega_2 = \alpha_2$. So, $\alpha = d(\omega_1 + \omega_2 + \omega)$, $I_n$ is acyclic, which proves that $V$ is acyclic.
\end{proof}

Together, these lemmas give us the following proposition.
\begin{proposition} \label{prop:phiquasiiso}
The map $\bar{\phi} \de D(\Ar)/D(\Ar') \to D(\widetilde{\Ar''})$ is a quasi-isomorphism.
\end{proposition}
\begin{proof}
By lemma~\ref{lem:vkernel}, we can factorize $\bar{\phi}$ as shown in the following diagram.
\[
\xymatrix{\frac{D(\Ar)}{D(\Ar')} \ar[r]^{\bar{\phi}} \ar[d]_p & D(\widetilde{\Ar''}) \\ \frac{D(\Ar)}{D(\Ar') \oplus V}  \ar[ru]_{\tilde{\phi}} & 
}\]
where $\tilde{\phi}$ is the quotient map.  Since $V = \ker \bar{\phi}$ and $\bar{\phi}$ is surjective, $\tilde{\phi}$ is an isomorphism. By lemma~\ref{lem:acyclic}, the map $p$ is a quasi-isomorphism, which implies that $\bar{\phi}$ is a quasi-isomorphism as well.
\end{proof}

\section{A long exact sequence in rational cohomology} \label{sec:drt}

The quasi-isomorphism $\bar{\phi} \de D(\Ar)/D(\Ar') \to D(\widetilde{\Ar''})$, described in section~\ref{sec:quasiiso}, decreases the degree by $2 \codim x_0 - 1 = \deg(x_0)$.  Therefore, we have the following theorem.
\begin{theorem} \label{theo:lesic}
Let $\Ar$ be a subspaces arrangement, $x_0 \in \Ar$ with $\Ar'$ and $\Ar''$ the deleted and the restricted arrangements. Then, there is a long exact sequence in cohomology ~: \begin{multline*}
\dotso \to H^q(M(\Ar'), \Q) \to H^q(M(\Ar), \Q) \to \\ H^{q-\deg(x_0)} (M(\Ar''), \Q) \to  H^{q+1}(M(\Ar'), \Q) \to \dotso
\end{multline*}
\end{theorem}
\begin{proof}
It is a direct consequence from the short exact sequence~$0 \to D(\Ar') \to D(\Ar) \to D(\Ar)/D(\Ar') \to 0$, corollary~\ref{cor:artilde} and proposition~\ref{prop:phiquasiiso}.
\end{proof}

As a first consequence, the Euler characteristic of the complement of non empty subspaces arrangements is zero.
\begin{corollary} \label{cor:euler}
Let $\Ar$ be a subspaces arrangement in $\C^l$. Then, the Euler characteristic of $M(\Ar)$ satisfies~: $\chi(M(\Ar)) = 0$ if $\Ar$ is non-empty, and $\chi(M(\Ar)) = 1$ if $\Ar$ is empty.
\end{corollary}
\begin{proof}
If $\Ar$ is empty, then $M(\Ar) = \C^l$ and $\chi(M(\Ar)) = 1$. If $|\Ar| = 1$, then $M(\Ar) = \C^l \setminus \{x_0\}$ for some vector subspace $x_0$.  So, $M(\Ar)$ has the homotopy type of an odd-dimensional sphere and $\chi(M(\Ar)) = 0$.

By the universal coefficient theorems (in homology and cohomology), we have the following equality : $b_i(M(\Ar)) = \dim H^i(M(\Ar), \Q)$, where $b_i(M(\Ar))$ is the $i$th Betti number of $M(\Ar)$.  The long exact sequence in cohomology from theorem~\ref{theo:lesic} gives us the following formula~: $\chi(M(\Ar)) = \chi(M(\Ar')) + (-1)^{\deg(x_0)}\chi(M(\Ar'')) = \chi(M(\Ar')) - \chi(M(\Ar''))$. Now, let's prove by induction on $|\Ar|$ that $\chi(M(\Ar)) = 0$. We suppose that this is true for arrangements with $n$ or less subspaces. Let $\Ar$ be an arrangement with $n+1$ subspaces and $x_0 \in \Ar$. The deleted and restricted arrangements are such that $|\Ar'| = |\Ar| - 1$ and $1 \leq |\Ar''| < |\Ar|$.  So, by induction, $\chi(M(\Ar)) = \chi(M(\Ar')) - \chi(M(\Ar'')) = 0 - 0 = 0$.
\end{proof}

The special case of arrangements of hyperplanes is well known (see~\cite{orl92}). If $\Ar$ is a complex arrangement of hyperplanes, not only do we know that $\chi(M(\Ar)) = 0$, but we also know that the Poincar\'e polynomials of a triple $(\Ar, \Ar', \Ar'')$ verify the relation~: $\Poin(M(\Ar), t) = \Poin(M(\Ar'), t) + t\Poin(M(\Ar''), t)$. Therefore, a natural question arises~: can we extend this formula to the general case of subspaces arrangements?  A generalized formula is equivalent to the splitting of the long exact sequence of theorem~\ref{theo:lesic} in short exact sequences \[
0 \to H^q(M(\Ar'), \Q) \to H^q(M(\Ar), \Q) \to \\ H^{q-\deg(x_0)} (M(\Ar''), \Q) \to  0.
\]
This would give us the following formula~: 
\begin{equation} \label{eq:formula}
\Poin(M(\Ar), t) = \Poin(M(\Ar'), t) + t^{\deg(x_0)}\Poin(M(\Ar''), t). 
\end{equation}
But in general, the equality~\ref{eq:formula} does not hold, as shown by the following example.

\emph{Example 1.} In $\C^5$, let's consider the arrangement $\Ar = \{h_0, h_1, h_2\}$ where $h_0, h_1$ and $h_2$ are defined by the equations $z_1 = z_5 = 0$, $z_1 = z_2 = z_3 = 0$ and $z_3 = z_4 = z_5 = 0$, respectively.

A basis of the cochain complex $D(\Ar)$ is described in table~\ref{tab:descar}. From that, it is easy to see that $\Poin(M(\Ar), t) = 1 + t^3 + 2t^5 + 2t^6$ (and that $\chi(M(\Ar)) = 0$). Let's consider $\Ar' = \Ar \setminus \{h_0\}$ and the corresponding restricted arrangement $\Ar''$.  Some quick computations yield $\Poin(M(\Ar'), t) = 1 + 2t^5 + t^8$, $\Poin(M(\Ar''), t) = 1 + 2t^3 + t^4$ and \[
\Poin(M(\Ar'), t) + t^3\Poin(M(\Ar''), t) = 1 + t^3 + 2t^5 + 2t^6 + t^7 + t^8.
\]

\begin{table}[hbt]
\begin{center}
\begin{tabular}{|c|c|c||c|c|c|}
\hline 
$\deg(x)$ & $x$ & $dx$ & $\deg(x)$ & $x$ & $dx$ \\
\hline
0 & $\{\emptyset\}$ & $d\{\emptyset\} = 0$ & 6 & $\{h_0, h_1\}$ & $d\{h_0, h_1\} = 0$\\
3 & $\{h_0\}$ & $d\{h_0\} = 0$ & 6 & $\{h_0, h_2\}$ & $d\{h_0, h_2\} = 0$ \\
5 & $\{h_1\}$ & $d\{h_1\} = 0$ & 7 & $\{h_0, h_1, h_2\}$ & $d\{h_0, h_1, h_2\} = - \{h_1, h_2\}$\\
5 & $\{h_2\}$ & $d\{h_2\} = 0$ & 8 & $\{h_1, h_2\}$ & $d\{h_1, h_2\} = 0$\\
\hline
\end{tabular}
\end{center}
\caption{Description of $D(\Ar)$ for example~1.}
\label{tab:descar}
\end{table}

So, the formula~\ref{eq:formula} isn't always true.
\begin{definition} \label{def:sep}
Let $\Ar$ be a subspaces arrangement in $\C^l$. We say that $x_0$ satisfies the Poincar\'e polynomial property (in short~: PP property) if for the corresponding triple $(\Ar, \Ar', \Ar'')$, we have \[
\Poin(M(\Ar), t) = \Poin(M(\Ar'), t) + t^{\deg(x_0)}\Poin(M(\Ar''), t). 
\]
\end{definition}

In the example, $h_0$ does not satisfy the PP property. However, if we consider the triple $(\Ar, \widehat{\Ar}', \widehat{\Ar}'')$ associated to $h_1$, then the Poincar\'e polynomials are~:
\begin{align*}
\Poin(M(\widehat{\Ar'}), t) &= 1 + t^3 + t^5 + t^6 \\
t^5 \Poin(M(\widehat{\Ar'}), t) &= t^5(1 + t) \\
\Poin(M(\Ar), t) &= 1 + t^3 + 2t^5 + 2t^6.
\end{align*}
The formula~\ref{eq:formula} is thus satisfied for $h_1$. This example illustrates that, unlike with hyperplanes arrangements, not every subspace $x_0$ has the PP property. But the arrangement $\Ar$ of the example shows that even if a given subspace has not the PP property, then sometimes, we can find another subspace which possesses that property. At this point, it seems natural to ask the following questions~:
\begin{enumerate}
\item From any (non-empty) subspaces arrangement $\Ar$, can we always choose a $x \in \Ar$ satisfying the PP property? 
\item Can we find simple equivalent conditions for $x \in \Ar$ to have the PP property?
\item For any arrangement of hyperplanes $\Ar$, each hyperplane has the PP property. Can we find a larger class of subspaces arrangements for which this is true? For example, is this true for arrangements with a geometric lattice?
\end{enumerate}

Notice that the arrangement $\Ar$ of example~1 does not have a geometric lattice. To simplify the notations, for $I = \{i_1, \dotsc, i_r\} \subseteq \{0, \dotsc, n\}$, we'll denote $x_{i_1} \cap \dotso \cap x_{i_r}$ by $x_{\{i_1, \dotsc, i_r\}}$, or sometimes simply $x_I$. The following definition is a generalization of the concept of separator in~\cite{orl92}.
\begin{definition}
Let $\Ar = \{x_0, \dotsc, x_n\}$ be a subspaces arrangement. We say that $x_0$ is a \emph{separator} if $x_0 \cap x_{\{1, 2, \dotsc, n\}} > x_{\{1, 2, \dotsc, n\}}$ (in the lattice $L(\Ar)$).
\end{definition}

Now, a partial answer to question~2 is given by the following proposition.
\begin{proposition} \label{prop:sep}
Let $\Ar$ be a subspaces arrangement in $\C^l$ and $x_0 \in \Ar$. If $x_0$ is a separator, then $x_0$ has the PP property
\end{proposition}
\begin{proof}
To prove that $x_0$ has the PP property, it is sufficient to show that the injective map $j \de D(\Ar') \to D(\Ar)$ induces an injective map $j^\star \de H^\star(D(\Ar')) \to H^\star(D(\Ar))$, since, in that case, the long exact sequence of theorem~\ref{theo:lesic} splits into short exact sequences \[
0 \to H^q(M(\Ar'), \Q) \to H^q(M(\Ar), \Q) \to \\ H^{q-\deg(x_0)} (M(\Ar''), \Q) \to  0.
\]
Let's define a map $k \de D(\Ar) \to D(\Ar')$ of cochain complexes by $k(\sigma) = 0$ if $x_0 \in \sigma$ and $k(\sigma) = \sigma$ if $x_0 \not\in \sigma$.  This map commutes with the differential~:
\begin{itemize}
\item If $x_0 \not\in \sigma$, then $dk(\sigma) = kd(\sigma)$,
\item if $x_0 \in \sigma$, then $dk(\sigma) = 0$ and $kd(\sigma) = \alpha (\sigma \setminus \{x_0 \})$, with $\alpha = 0$ if and only if $\cap_{x \in \Ar} x \neq \cap_{y \in \Ar'} y$, which is the case.
\end{itemize}
So, the map $k$ is a map of chain complexes and induces a map $k^\star \de H^\star(D(\Ar)) \to H^\star(D\Ar'))$. Obviously, we have $k^\star j^\star = \id$, which implies that $j^\star$ is injective.  Hence $x_0$ has the PP property.
\end{proof}

In the example~1, we can use the proposition~\ref{prop:sep} to check that $h_1$ and $h_2$ are separators. But question 2 is not completely answered~: the following example is an arrangement $\Br$ which has the PP property with respect to $h_0$. However, $h_0$ is not a separator.

\emph{Example 2.} In $\C^3$, let's consider $\Br = \{h_0, h_1, h_2\}$ where $h_0, h_1$ and $h_2$ are the three lines defined by the equations $z_2 = z_3 = 0$, $z_1 = z_3 = 0$ and $z_1 = z_2 = 0$, respectively. A basis for $D(\Br)$ is $1$ in degree 0, $X_i = \{h_i\}$ in degree 3, $X_{ij} = \{h_i, h_j\}$ in degree 4 and $X_{012}$ in degree 3, with $d1 = dX_i = dX_{ij} = 0$ and $dX_{012} = -X_{12} + X_{02} - X_{01}$.  From that, it is easy to compute 
\begin{align*}
\Poin(M(\Br), t) &= 1 + 3t^3 + 2t^4 = 1 + 2t^3 + t^4 + t^3(1 + t) \\
&= \Poin(M(\Br'), t) + t^{\deg(h_0)} \Poin(M(\Br''), t).
\end{align*}

\section{Arrangements with a geometric lattice} \label{sec:agl}

When we suppose that the arrangement $\Ar$ has a geometric lattice, we have some stronger results. These results come from the following lemma.
\begin{lemma} \label{lem:prep}
Let $\Ar = \{x_0, \dotsc, x_n\}$ be an arrangement with a geometric lattice such that $x_0$ is a separator. Let $\{i_1, \dotsc, i_{r+1}\} \subset \{1, \dotsc, n\}$. Then we have \[
x_0 \cap x_{\{i_1, \dotsc, i_r\}} = x_0 \cap x_{\{i_1, \dotsc, i_{r+1}\}} \text{ if and only if } x_{\{i_1, \dotsc, i_r\}} = x_{\{i_1, \dotsc, i_{r+1}\}}.
\]
\end{lemma}
\begin{proof}
If $x_{\{i_1, \dotsc, i_r\}} = x_{\{i_1, \dotsc, i_{r+1}\}}$, then obviously, we have $x_0 \cap x_{\{i_1, \dotsc, i_r\}} = x_0 \cap x_{\{i_1, \dotsc, i_{r+1}\}}$. 

Suppose that $x_0 \cap x_{\{i_1, \dotsc, i_r\}} = x_0 \cap x_{\{i_1, \dotsc, i_{r+1}\}}$. From $x_0 \cap x_{\{1, \dotsc, n\}}  \neq x_{\{1, \dotsc, n\}}$, we deduce that $x_0 \cap x_{\{i_1, \dotsc, i_r\}}  \neq x_{\{i_1, \dotsc, i_r\}}$.  Since the lattice is geometric, we obtain the maximal chain $x_{\{i_1, \dotsc, i_r\}} < x_0 \cap x_{\{i_1, \dotsc, i_r\}}$ from the maximal chain $\C^l < x_0$. Therefore, we have the two following chains of inequalities~: 
\begin{align*}
x_{\{i_1, \dotsc, i_r\}}  < x_0 \cap x_{\{i_1, \dotsc, i_r\}} & = x_0 \cap x_{\{i_1, \dotsc, i_{r+1}\}}, \\
x_{\{i_1, \dotsc, i_r\}}  \leq x_{\{i_1, \dotsc, i_{r+1}\}} & < x_0 \cap x_{\{i_1, \dotsc, i_{r+1}\}}.
\end{align*}
Since the lattice is geometric, all maximal chains between two elements have the same length. The first line is a maximal chain of length 2, so the second line has the same length and we have the equality $x_{\{i_1, \dotsc, i_r\}} = x_{\{i_1, \dotsc, i_{r+1}\}}$.
\end{proof}

Suppose that the arrangement $\Ar$ satisfies the assumptions of lemma~\ref{lem:prep} and let's define a map $\theta \de D(\Ar') \to D(\widetilde{\Ar''})$ by sending $\{x_{i_1}, \dotsc, x_{i_r}\}$ to $\{x_0 \cap x_{i_1}, \dotsc, x_0 \cap x_r\}$. This map is a bijection~: if $x_0 \cap x_i = x_0 \cap x_j$, then $x_0 \cap x_i = x_0 \cap x_{\{i,j\}}$, and by lemma~\ref{lem:prep}, $x_i = x_{\{i,j\}}$, which is impossible. 

Lemma~\ref{lem:prep} implies that $\theta$ commutes with the differential. So, the induced map in cohomology $H^\star \theta \de H^\star(M(\Ar'), \Q) \to H^\star(M(\Ar''), \Q)$ is an isomorphism of vector spaces.

But this map $H^\star \theta$ does not preserve the multiplication, or even the grading. A simple computation shows that the element $x = \{x_{i_1}, \dotsc, x_{i_r}\}$ of degree $\deg(x)$ (in $H^\star(M(\Ar'), \Q)$) is sent to $\{x_0 \cap x_{i_1}, \dotsc, x_0 \cap x_{i_r}\}$, which has a degree $\deg(x) - 2\codim(x_0+x_{\{i_1, \dotsc, i_r\}})$ (in $H^\star(M(\Ar''), \Q)$). So, it decreases the degree by an even number. This proves the following proposition~: 
\begin{proposition}
Let $\Ar = \{x_0, \dotsc, x_n\}$ be an arrangement with a geometric lattice such that $x_0$ is a separator. Then we have the following relations~:
\begin{align*}
&\sum_{i=0}^\infty b_{2i}(M(\Ar')) = \sum_{i=0}^\infty b_{2i}(M(\Ar'')),  && \sum_{i=0}^\infty b_{2i+1}(M(\Ar')) = \sum_{i=0}^\infty b_{2i+1}(M(\Ar'')), \\
&\sum_{i=0}^\infty b_{i}(M(\Ar')) = \sum_{i=0}^\infty b_{i}(M(\Ar'')), && \forall n \in \N, \sum_{i=0}^n b_{i}(M(\Ar'))  \leq \sum_{i=0}^n b_{i}(M(\Ar'')).
\end{align*}
\end{proposition}

Finally, when the subspace $x_0$ is an hyperplane, we have a stronger result. 
\begin{proposition}
Let $\Ar = \{x_0, \dotsc, x_n\}$ be an arrangement with a geometric lattice such that $x_0$ is an hyperplane and a separator, then the graded vector spaces $H^\star(M(\Ar'), \Q)$ and $H^\star(M(\Ar''), \Q)$ are isomorphic and 
\begin{align*}
\Poin(M(\Ar'), t) &= \Poin(M(\Ar''), t),\\
\Poin(M(\Ar), t) &= (1+t) \Poin(M(\Ar''), t). 
\end{align*}
\end{proposition}
\begin{proof}
Since $x_0$ is an hyperplane and $x_0 \cap x_{\{1, \dotsc, n\}}  \neq x_{\{1, \dotsc, n\}}$, we have $\codim(x_0+x_{\{i_1, \dotsc, i_r\}}) = 0$. Therefore, the map $H^\star \theta$ preserve the degree and is an isomorphism of graded vector spaces. This fact and proposition~\ref{prop:sep} implies directly the two equalities.
\end{proof}

\end{document}